\documentclass{amsart}

\usepackage{amssymb}
\usepackage{amsfonts}

\usepackage[mathscr]{euscript}

\input xy
\xyoption{all}

\usepackage{natbib}

\newtheorem*{theorem*}{Theorem}
\newtheorem{theorem}{Theorem}
\newtheorem{corollary}[theorem]{Corollary}
\newtheorem{definition}{Definition}
\newtheorem{lemma}{Lemma}
\newtheorem*{remark*}{Remark}

\newcommand{\arrow}[3]{\ensuremath{#1\colon#2\rightarrow#3}} 
\newcommand{\marrow}[3]{\ensuremath{#1\colon#2\rightarrowtail#3}}

\newcommand{\clarrow}[2]{\ensuremath{#1\rightarrow#2}}
\newcommand{\clmarrow}[2]{\ensuremath{#1\rightarrowtail#2}}

\newcommand{\cprod}[2]{\ensuremath{#1\times #2}} 
\newcommand{\cpull}[3]{\ensuremath{#1\times_{#2} #3}} 

\newcommand{\demph}[1]{\textbf{#1}} 

\newdir{ >}{{}*!/-10pt/@{>}} 

\mathchardef\mhyphen="2D

\usepackage[mathscr]{euscript}

\title{A univalent universe in finite order arithmetic}
\author{Colin McLarty}

\begin{document}

\begin{abstract}Homotopy Type Theory with a univalent universe $\,\mathcal{U}_0$ is interpreted at the strength of finite order arithmetic.  We eliminate Grothendieck universes, avoid the axiom of replacement, and bound all uses of separation.
\end{abstract}

\maketitle

Kapulkin, Lumsdaine, and Voevodsky \citeyearpar[Thm.~3.4.3]{KapuLenFanuVoev} show the Univalence Axiom is consistent with Martin-L\"of type
theory assuming the existence of two inaccessible cardinals in ZFC.  We remove the inaccessibles by using a definability criterion for \emph{smallness} rather than cardinality.

More fully:  Section~\ref{S:finiteorder} interprets a univalent universe $\,\mathcal{U}_0$ at the logical strength of finite order arithmetic, using \emph{Mac~Lane Type Theory} (MTT)\@.  The $\,\mathcal{U}_0$-small types admit all the constructors $\Pi\mhyphen,\Sigma\mhyphen,\mathrm{Id}\mhyphen,\mathrm{W}\mhyphen,\mathbf{1},\mathbf{0},$ and $+$ of Martin-L\"of type theory.  Non-small types in this interpretation admit all but $\Pi\mhyphen$.  In general MTT is a conservative third order extension of set theory~\citep{McLArXivLarge}.  The version here is conservative over the set theory \textsf{MAC}+V=L introduced by \citet{MathThe}.  The set axioms are ZFC, minus replacement, with bounded separation, transitive containment and V=L\@.  Regularity and transitive containment give a good theory of rank while V=L gives a definable well ordering of all sets. This MTT, like \textsf{MAC} itself, is equiconsistent with finite order arithmetic.

As a quick guide to the limitations of MTT note the set theoretic axioms have a natural model in $L_{\omega+\omega}$.  So no set can be proved to exist with rank infinite over $\mathbb{N}$.  Even the defining conditions for sets must be bounded.  MTT cannot prove there exists a set of natural numbers
  \[\{n\in\mathbb{N}|\ \text{ there are } n \text{ successive power sets beginning with } \mathbb{N} \}.\]

\subsection{A concern} Steve Awodey at the ASL in Philadelphia noted that version did  not address coherence.  I believed use of a (not-necessarily-univalent) universe following \cite[\S 1]{KapuLenFanuVoev} would carry over routinely.  I now think Steve is right this belief needs more careful verification.  If it does not work, then the best option seems to be extending the method of our new Section~\ref{SS:coherence} to classes by defining a well ordering of classes using a higher consistency strength than finite order arithmetic, namely second order Zermelo set theory.  The method would combine results by Cohen, Felgner and Mostowski with the assumption V=L for sets (see \citet{SimpsonRevFelgner}, thanks to Victoria Gitman for pointing out these results).  I have not yet worked out the details either way.

\section{Quillen model structure on simplicial sets in \textsf{MAC}}
\subsection{Working with bounded separation}\label{SS:separation}
A few germane examples contrast bounded separation to the axiom of replacement.

Define objects of the simplex category $\Delta$ as initial segments $[n]=\{0,\dots,n\}$ of $\mathbb{N}$, and $\mathrm{Hom}_{\Delta}([n],[m])$ as the set of order preserving functions \clarrow{[n]}{[m]}.  Countable replacement could collect the $[n]$ into the object set $\Delta_0=\{[n]\,|\ n\in \mathbb{N}\,\}$, and collect the  $\mathrm{Hom}_{\Delta}([n],[m])$ into a set whose union is the arrow set $\Delta_1$.  But the definitions only refer to elements and subsets of $\mathbb{N}$ and $\cprod{\mathbb{N}}{\mathbb{N}}$\@.  So $\Delta_0$ can be defined by bounded separation instead of replacement:
  \[ \Delta_0=\{x\in \mathcal{P}(\mathbb{N})\,|\ \exists n\in \mathbb{N}\, \forall y\in\mathbb{N}\,(y\in x \leftrightarrow y<n)\ \}\]
And separation defines $\Delta_1\subseteq \mathcal{PP}(\cprod{\mathbb{N}}{\mathbb{N}})$ with equally natural quantifier bounds.

And we cannot define a presheaf $\mathcal{F}$ on small category $\mathcal{C}$ as a set $\mathcal{F}A$ for each $A\in \mathcal{C}_0$ and suitable functions \arrow{{\mathcal{F}f}}{\mathcal{F}B}{\mathcal{F}A} for $\mathcal{C}$ arrows \arrow{f}{A}{B}. In the absence of replacement, this does not imply any set bounds the values $\mathcal{F}A$.

So define a presheaf $\mathcal{F}$ on small category $\mathcal{C}$ as a function \arrow{\mathcal{F}}{F_0}{\mathcal{C}_0}, and a function \arrow{\gamma_{\mathcal{F}}}{\cprod{\mathcal{C}_1}{F_0}}{F_0} meeting appropriate conditions.  Then $\mathcal{F}(A)$ for $A\in \mathcal{C}_0$ is the set $\mathcal{F}^{-1}(A)\subseteq F_0$ so $F_0$ bounds the values while $\gamma_{\mathcal{F}}$ combines all the functions \arrow{{\mathcal{F}f}}{\mathcal{F}B}{\mathcal{F}A}.  The full definition of this \emph{Grothendieck construction} is in many references including \cite[Section~3.2]{McLArXivLarge}.

A \emph{simplicial set} is a presheaf on the simplex category $\Delta$. So it is a pair of functions \arrow{X}{X_0}{\Delta_0} and \arrow{\gamma_{X}}{\cprod{\Delta_1}{X_0}}{{X_0}} meeting the presheaf conditions. Define the set $X[n]$ of $n$-simplices in a simplicial set $X$ as the preimage $X[n]=X^{-1}([n])$ of object $[n]\in\Delta_0$ along the function  \arrow{X}{X_0}{\Delta_0}.

A \emph{simplicial set map} \arrow{f}{Y}{X} is a natural transformation of presheaves.  By Yoneda $X[n]$ amounts to the set of simplicial maps \clarrow{\Delta[n]}{X} where $\Delta[n]$ is the standard $n$-simplex, formally the represented presheaf  $\mathrm{Hom}_{\Delta}(\_,[n])$.

\subsection{Coherence for sets}\label{SS:coherence}
The \emph{coherence problem} for categorical logic is ``the requirement for pullback to be strictly functorial, and for the logical structure[s] to commute strictly with it''~\cite[\S 1.3]{KapuLenFanuVoev}.  Global choice can make pullback strictly functorial in one argument:

\begin{definition}For any functions \arrow{f}{A}{B} and \arrow{h}{C}{B} the \demph{selected pullback} $f^*(h)$ is the smallest function in the global well ordering which makes a pullback.
\end{definition}

\[\xymatrix{f^*(C) \ar[r] \ar[d]_{f^*(h)} & C \ar[d]^h \\  A \ar[r]_f & B}\]

\begin{lemma} Selected pullbacks are strictly functorial. So $g^*f^*(h)=(fg)^*(h)$ here:
\[\xymatrix{\bullet \ar[r] \ar[d]  & f^*(C) \ar[r] \ar[d]_{f^*(h)} & C \ar[d]^h \\  D \ar[r]_g & A \ar[r]_f & B}\]
\end{lemma}
\begin{proof} The familiar fact that when the right hand square is a pullback, then the left hand square is a pullback iff the outer rectangle is.
\end{proof}

Categorical logic  makes asymmetric use of pullbacks.  The function $h$ is regarded as some logical construct over some base $B$, while $f$ is a re-parametrization or change of base from $B$ to $A$.  Pullback is required to be strictly functorial in change of base, and this definition makes it so.

And this gives a routine for making all logical structures commute strictly with pullback.  We convene that each logical construction, which is prima facie defined up to isomorphism over the base, will be specified to be the unique smallest function to the base in that isomorphism class.  Trivially, pullback along base changes is strictly functorial for these specified representatives.

For example given functions \arrow{k}{K}{A} and \arrow{f}{A}{B} the dependent sum function \arrow{\Sigma(k)}{\Sigma(K)}{B} is commonly defined up to isomorphism by a universal property making $\Sigma$ left adjoint to pullback along $f$. And it is commonly specified to be just the postcomposite with $f$ so $\Sigma(K)=K$ and $\Sigma(k)=fk$.  But our convention will rather specify  $\Sigma(k)$ as that function to $B$ that comes earliest in the global well order among all functions isomorphic over $B$ to $fk$, and  $\Sigma(K)$ to be whatever set is the domain of that function.

We define selected pullbacks, and logical constructors, for simplicial sets the same way.  So we have coherence for type theory on simplicial sets.

\subsection{W-types}\label{SS:wtypes}

A W-type in sets is a set generated by certain elements and operators, subject to no equations.  For example the set of natural numbers is generated by $0$ with a unary successor operator $s$, so $\mathbb{N}=\{0,s0,ss0\dots\}$.  Constants are treated as 0-ary operators.  A set of operators, each with an arity, is called a \demph{signature}.  Since there are no equations an operator has no significant properties but its arity, and in this strict sense the signature of the natural numbers as W-type is just one 0-ary operator and one unary.

It is often useful to treat a W-type as an infinite colimit and that could seem to use replacement.  But each W-type can also be realized by a suitable set of finite sequences on a given alphabet defined by the operators and their arities \citep[p.~3]{vandenBergMoerdijk}.  Since \textsf{MAC} provides for every set $S$ a set $S^{\mathbb{N}}$ of all infinite sequences of elements of $S$, that approach works in \textsf{MAC} to prove every signature has a W-type.  Section~\ref{S:classes} treats W-types in the class theory MTT.

\subsection{Fibrations}
The chief logical issue in verifying the Quillen model structure on simplicial sets is existence of cofibration-fibration factorizations.

\begin{theorem}\label{T:fibration} \emph{(Provably, in \textsf{MAC})} Any simplicial map \arrow{f}{X}{Y} can be factored into an acyclic monic \marrow{j}{X}{Z} and fibration \arrow{p}{Z}{Y}, or into a monic \clmarrow{X}{Z} and an acyclic fibration \clarrow{Z}{Y}.
\end{theorem}
\begin{proof}The two claims are proved the same way.  We handle the first since it is used for homotopy identity types. Given  \arrow{f}{X}{Y} Quillen first adds simplices to $X$ to form a simplicial set $Z^0$, then adds simplices to that to form $Z^1$, and so on through an infinite series.  The colimit of all the $Z^i$ is the desired $Z$.  Replacement would allow us to form each $Z^i$ successively and then collect them into a colimit $Z$.  To do it in $\mathsf{MAC}$, though, using bounded separation and not replacement, we must use the resources of $\mathsf{MAC}$ to find a single set with enough structure to supply simplices for all the $Z^i$ before we actually find the $Z^i$.

We can choose as bounding set the W-type of the signature with one constant for each simplex of $X$ and two $n$-ary operators $\phi_n$ and $\kappa_n$ for each $n\in\mathbb{N}$.  Most elements of this W-type will not be used to construct the $Z^i$.  But we identify the constants with the corresponding simplices of $X$.  And when $\sigma_1,\dots,\sigma_n$ is any list of simplices forming an $n-1$-dimensional horn in $X$ they we treat the value $\phi(\sigma_1\dots\sigma_n)$ as an $n$-simplex filling that horn in $Z^0$; and treat $\kappa(\sigma_1\dots\sigma_n)$ as the last face of that new simplex in $Z^0$.  The infinite series $Z^0,Z^1,Z^2 \dots$ is defined inductively with all quantifiers bounded by the W-type.
\end{proof}

\subsection{Homotopy Type Theory}
Set theory \textsf{MAC} proves presheaves on any small category form a topos~\citep{McLArXivLarge} and so admit type theoretic constructors $\Pi\mhyphen,\Sigma\mhyphen,\mathbf{1},\mathbf{0}$, and~$+$; and following \citep{vandenBergMoerdijk} they have all W-types.  So \textsf{MAC} proves simplicial sets admit all these constructors.

\citet[Thm.~2.3.4]{KapuLenFanuVoev} show those simplicial set constructors preserve fibrations, and refer to \citet{AwodeyWarren} showing fibrations admit a homotopy-suitable identity type constructor $\mathrm{Id}\mhyphen$.  Given our Theorem~\ref{T:fibration} those proofs use no replacement and only bounded separation, so:

 \begin{theorem}\emph{\textsf{MAC}} proves fibrations as types over simplicial sets interpret Homotopy Type Theory with constructors $\Pi\mhyphen,\Sigma\mhyphen,\mathrm{Id}\mhyphen, \mathrm{W}\mhyphen, \mathbf{1},\mathbf{0}, +$.
 \end{theorem}

It remains to get these constructors for simplicial proper-classes, where dependent types have set-sized fibers.

\section{Classes and collections}\label{S:classes}

\emph{Mac~Lane Type Theory} (MTT) takes sets, classes, and collections as types. See~\citet{McLArXivLarge}. Keeping $\in$ for set membership, write $A\in^1 \mathcal{A}$ to say set $A$ is in class $\mathcal{A}$, and $\mathcal{A}\in^2 \mathfrak{B}$ for class $\mathcal{A}$ in collection $\mathfrak{B}$.  MTT defines classes and collections by \emph{set theoretic formulas}, that is formulas quantifying only over sets.  When MTT assumes the axioms \textsf{MAC}+V=L for sets it is conservative over \textsf{MAC}+V=L, as G\"odel-Bernays set theory is over ZFC\@.

For example, class inclusion, cartesian product of classes, and the collection of all functions between classes $\mathcal{A,B}$\, are defined by set theoretic formulas:
\begin{center}
    $\mathcal{A}\subseteq^1\mathcal{B}\quad \leftrightarrow\quad \forall x\, (x\in^1\mathcal{A} \rightarrow x\in^1\mathcal{B}\,)$\\
    $\cprod{\mathcal{A}}{\mathcal{B}}\quad =\quad \{z\ |\ \exists\,x,y\ (z=\langle x,y\rangle\ \land\ x\in^1\mathcal{A}\ \land\ y\in^1 \mathcal{B}\,)\, \}^1$ \\
    $\mathcal{B^A}\ =\ \{\mathcal{F}\ |\ \mathcal{F}\subseteq^1 \cprod{\mathcal{A}}{\mathcal{B}}\
        \land\  (\forall x\in^1\mathcal{A})(\exists! y\in^1\mathcal{B})\ \langle x,y\rangle \in^1 \mathcal{F}   \}^2$
\end{center}

\noindent And the well ordering on sets given by V=L is defined by a set theoretic formula. All these concepts can be used to define classes and collections in MTT\@.

The theory MTT has tuple types.  There is a type of pairs $\langle\mathcal{A,B}\rangle^1$ of classes, and one of triples of classes $\langle\mathcal{A,B,C}\rangle^1$, and one of pairs $\langle\mathcal{A},\mathfrak{B}\rangle^2$ of one class and one collection, etc.  The superscript indicates the maximal type of an entry in the tuple.  There are collections of $n$-tuples of classes for any fixed $n$.

It is clear how to define the class category $\mathscr{S}\mathrm{et}=\langle\mathscr{S}\mathrm{et}_0,\mathscr{S}\mathrm{et}_1\rangle^1$ with the classes $\mathscr{S}\mathrm{et}_0$ of all sets and $\mathscr{S}\mathrm{et}_1$ of all functions.  Let $\mathfrak{Class}_0$ be the collection of all classes.  Code a class function \arrow{f}{\mathcal{A}}{\mathcal{B}} as a triple $\langle\mathcal{A},\Gamma_f,\mathcal{B}\rangle$ with $\Gamma_f\subseteq\cprod{\mathcal{A}}{\mathcal{B}}$ such that:
   \[\forall x\in^1 \mathcal{A}\ \exists_!\,y\in^1 \mathcal{B}\ \langle x,y\rangle\in^1 \Gamma_f \]
This set theoretic formula defines a collection $\mathfrak{Class}_1$ of all class functions.  The chief point in showing $\langle\mathfrak{Class}_0,\mathfrak{Class}_1\rangle^2$ is a collection category of all classes, is to show function composition is comprised in one collection
   \[ \mathfrak{Class}_2\ =\ \{\langle\mathcal{F,G,H}\rangle^1\ |\ \mathcal{GF=H}\,\}^2\]
It is so, since the formula $\mathcal{GF=H}$ only quantifies over sets in $\mathcal{F,G,H}$.

\begin{definition}\label{D:powerclass} Call $A$ a \demph{subset of class} $\mathcal{A}$, written $A\subseteq^0 \mathcal{A}$, if they satisfy the set theoretic formula
       $\forall x (x\in A \rightarrow x\in^1\mathcal{A})$. So every class $\mathcal{B}$ has a class $\mathcal{S}ub\mathcal{S}et(\mathcal{B}) = \{\, A\,|\, A\subseteq^0 \mathcal{B}\,)\,\}^1$ of all its subsets.
\end{definition}

\begin{remark*}Crucially for our project, MTT does not prove every subclass of a set is a set.  That principle amounts to Zermelo's axiom of separation.  It proves the consistency of finite order arithmetic, and so of MTT.
\end{remark*}

Classes in MTT differ from sets not just in size, but in that sets require bounded definitions entirely in the language of \textsf{MAC}.  Classes only need set theoretic definitions allowing class and collection parameters.

\begin{theorem}\label{T:exactness} The category $\mathfrak{Class}$ has all finite limits and coproducts, and has quotients for equivalence relations defined by set theoretic formulas.
\end{theorem}
\begin{proof}The usual products and pullbacks work. Define disjoint unions by
    \[\mathcal{A+B}=\ \{\langle x,y\rangle |\ (x\in^1\mathcal{A}\ \land\ y=0) \vee (x\in^1\mathcal{B}\ \land\ y=1)   \}^1\]
For a quotient of any equivalence relation $\mathcal{E}\subseteq^1\cprod{\mathcal{A}}{\mathcal{A}}$ defined by a set theoretic formula take the class of all $y\in^1 \mathcal{A}$ minimal in the global well ordering of sets among all $x$ with $x\mathcal{E}y$.
\end{proof}

\begin{definition}A class function \arrow{f}{\mathcal{A}}{\mathcal{B}} is \demph{bounded} if every $A\subseteq^0 \mathcal{A}$ has a set image $f(A)\subseteq^0 \mathcal{B}$ with the graph of the restriction $f|_{A}$ a subset of $\cprod{A}{f(A)}$ (not merely a subclass).
\end{definition}

Boundedness is a definability condition, not isomorphism invariant in $\mathfrak{Class}$.  But bounded functions are closed under composition to they define a subcategory  $\mathfrak{BClass}$ which also has good categorical properties:

\begin{lemma}The projections \clarrow{\cprod{\mathcal{A}}{\mathcal{B}}}{\mathcal{A}} and  \clarrow{\cprod{\mathcal{A}}{\mathcal{B}}}{\mathcal{B}} from the canonical product of two classes are bounded, as are canonical class equalizers, and pullbacks.  So are the injections \clmarrow{\mathcal{A}}{\mathcal{A+B}} and \clmarrow{\mathcal{B}}{\mathcal{A+B}} into a canonical disjoint union.
\end{lemma}
\begin{proof}The \textsf{MAC} axioms within MTT imply any set of ordered pairs has a set of first components.  The case of equalizers is trivial since $E\subseteq^0 \mathcal{E}\subseteq^1 \mathcal{A}$ implies $E\subseteq^0 \mathcal{A}$.  Canonical pullbacks are equalizers of canonical products.
\end{proof}

\begin{theorem}\label{T:cartesianfragment}
For any set $A$ and class $\mathcal{B}$ the set theoretic formula
  \[\exists B\subseteq^0\mathcal{B},\ \arrow{f}{A}{B}\]
defines the class $\mathcal{B}^A$ of all bounded functions from the class $\{x | x\in A\}^1$ to $\mathcal{B}$.
\end{theorem}

\begin{definition}
A class function \arrow{f}{\mathcal{C}}{\mathcal{B}} is \demph{locally small} if every subset $B\subseteq \mathcal{B}$ has a subset pre-image $f^{-1}(B)\subseteq^0 \mathcal{C}$ with the graph of the restriction a subset of $\cprod{f^{-1}(B)}{B}$.
 \end{definition}

 \begin{lemma} In $\mathfrak{BClass}$, all pullbacks of locally small \arrow{f}{\mathcal{C}}{\mathcal{B}} are locally small.
 \end{lemma}
 \begin{proof}For $A\subseteq^0 \mathcal{A}$ take sets $B=g(A)$, $C=f^{-1}(B)$ and set functions $g|_A$, $f|_C$:
  \[\xymatrix@R=1ex{ \cpull{A}{B}{C}  \ar[dd] \ar[r] & C \ar[dd]^>>>>>>{f|_C}  && \mathcal{\cpull{A}{B}{C}} \ar[r] \ar[dd] & \mathcal{C} \ar[dd]^>>>>>>f\\  && \subseteq^0  \\
                A \ar[r]_{g|_A} & B   &&  \mathcal{A} \ar[r]_>>>>>>>g & \mathcal{B}}\]
 \vspace{-6ex}

 \hfill\qedhere
  \end{proof}

Local smallness resembles the idea of $\alpha$-small in \citeyearpar{KapuLenFanuVoev} but is not just a size bound on fibers.  It is a definability condition on sets of fibers.  And it gives a crucial fragment of local cartesian closedness for bounded class functions:

\begin{theorem} In $\mathfrak{BClass}$, pullback $\mathcal{F}^*$ along \arrow{\mathcal{F}}{\mathcal{B}}{\mathcal{A}} has a left adjoint $\Sigma_{\mathcal{F}}$.  And if $\mathcal{F}$ is  locally small there is a right adjoint $\Pi_{\mathcal{F}}$.
 \end{theorem}
 \begin{proof}Straightforward verification.
 \end{proof}

\begin{corollary}  In $\mathfrak{BClass}$, every locally small class function \arrow{\mathcal{F}}{\mathcal{B}}{\mathcal{A}} has a polynomial functor $\mathcal{P_F}$ (compare  \citet[Def.~2.1]{vandenBergMoerdijk}).
  \[ \xymatrix@C=1ex{ \mathcal{P_F}& = &  \mathfrak{BClass} \ar[rrr]^<<<<<{\cprod{\_}\mathcal{B}} &&&  \mathfrak{BClass}/\mathcal{B} \ar[rrr]^<<<<<{\Pi_{\mathcal{F}}} &&&
                 \mathfrak{BClass}/\mathcal{A} \ar[rrr]^<<<<<{\Sigma_{\mathcal{A}}}  &&& \mathfrak{BClass} }\]
\end{corollary}

A W-type $\mathrm{W}(\mathcal{F})$ for locally small class function \arrow{\mathcal{F}}{\mathcal{B}}{\mathcal{A}} is an initial algebra for the functor $\mathcal{P_F}$.  Local smallness in this corollary means the signature for an algebra in MTT can have a proper class of operators, so long as each operator has a \emph{set} arity.  Each locally small functor has a W-type (unique up to isomorphism) which can be constructed as the class of all well-founded trees as it is in sets.  The initial algebra property is provable by well-founded induction on each tree, since the trees are sets.  Further, there is a class $\cprod{\mathcal{C}}{\mathbb{N}}$ where the slice $\mathcal{C}_n$ over any $n\in\mathbb{N}$ is the $n$-th iterate $\mathcal{P}_{\mathcal{F}}^n(\mathcal{\emptyset})$.  When all operators of the algebra are finitary, or in other words every fiber of  \arrow{\mathcal{F}}{\mathcal{B}}{\mathcal{A}} is finite, then by K\"onig's lemma all the trees in $\mathrm{W}(\mathcal{F})$ are finite, and MTT proves  $\mathrm{W}(\mathcal{F})$ is the colimit of these classes.

\section{Simplicial classes}\label{S:simplicialclasses}
\subsection{Quillen model structure on simplicial classes}
\begin{definition} \quad
\begin{enumerate}
   \item A \demph{simplicial class} is a class presheaf $\langle\mathcal{X},\gamma_{\mathcal{X}}\rangle$ on $\Delta$.  So \arrow{\mathcal{X}}{\mathcal{X}_0}{\Delta_0} and \arrow{\gamma_{\mathcal{X}}}{\cprod{\Delta_1}{\mathcal{X}_0}}{{\mathcal{X}_0}} meet the presheaf conditions.
   \item A \demph{simplicial class map} \arrow{f}{\mathcal{Y}}{\mathcal{X}} is a bounded natural transformation, a bounded function \arrow{f}{\mathcal{Y}_0}{\mathcal{X}_0} meeting the obvious conditions.
\end{enumerate}
 \end{definition}

So MTT proves there is a collection category $\mathrm{s}\mathscr{C}\mathrm{lass}$ of all simplicial classes and simplicial maps. It clearly has initial and terminal objects, products, and coproducts, and pullbacks.  The notions of a simplicial class map \arrow{f}{\mathcal{X}}{\mathcal{Y}} being a Kan fibration, or acyclic, are set theoretic, quantifying only over elements of $\mathcal{X}_0$ and finite subsets of $\mathcal{Y}_0$.  So they can be used to define classes, and there is a collection of all (acyclic) Kan fibrations in $\mathrm{s}\mathscr{C}\mathrm{lass}$.

\begin{theorem}\label{T:classfibration} \emph{(Provably, in MTT)} Any simplicial class map \arrow{f}{\mathcal{X}}{\mathcal{Y}} can be factored into an acyclic monic \marrow{j}{\mathcal{X}}{\mathcal{Z}} and a fibration \arrow{p}{\mathcal{Z}}{\mathcal{Y}}; and into a monic followed by an acyclic fibration.
\end{theorem}
\begin{proof} Adapt the W-type construction in Theorem~\ref{T:fibration} from sets to  $\mathfrak{BClass}$.  All the operators are finitary.
\end{proof}

\subsection{Homotopy Type Theory in  $\mathrm{s}\mathscr{C}\mathrm{lass}$}
What \cite{KapuLenFanuVoev} say on the constructors $\Sigma\mhyphen,\mathbf{1},\mathbf{0},+$ in $\mathrm{s}\mathscr{S}\mathrm{et}$ applies virtually unchanged in $\mathrm{s}\mathscr{C}\mathrm{lass}$, now that we know the relevant classes are well defined.  They refer to \citet{AwodeyWarren} for the identity type constructor.  Theorem~3.1 and Corollary~3.2 of that paper show how to define $\mathrm{Id}\mhyphen$ for $\mathrm{s}\mathscr{C}\mathrm{lass}$, because our Theorem~\ref{T:cartesianfragment} implies each simplicial class $\mathcal{A}$ has exponential $\mathcal{A^I}$ to serve as loop space where $\mathcal{I}=\Delta[1]$ is the usual simplicial  interval.

In general there is no class of all functions from one class to another.  The constructor $\Pi\mhyphen$ is interpretable for locally small dependent types, but not for arbitrary dependent types.

\section{One univalent universe in MTT}\label{S:finiteorder}

Now we can eliminate the inaccessible cardinal $\alpha$ from the construction of a univalent universe $\mathcal{U}_0$ by \citet[Section~2.1]{KapuLenFanuVoev}. Where they speak of sets with cardinality $<\alpha$ we speak of sets.  Where they speak of $\alpha$-small functions we speak of locally small.  For their non-small sets and functions we use classes and class functions.  We can make pullback of dependent types strictly functional by taking one (globally) selected representative of each.  And where they use isomorphism classes of $\alpha$-small Kan fibrations, we can use (globally) selected representative of each.  Apart from those technicalities the construction can follow theirs.

This does \emph{not} change their use of the category of arbitrarily well ordered  $\alpha$-small Kan fibrations and order preserving maps.  Those well orderings serve to eliminate non-identity automorphisms of extensions, so as to give a strictly universal family.

\bibliographystyle{apalike}

\begin{thebibliography}{}

\bibitem[Awodey and Warren, 2009]{AwodeyWarren}
Awodey, S. and Warren, M.~A. (2009).
\newblock Homotopy theoretic models of identity types.
\newblock {\em Math. Proc. Cambridge Philos. Soc.}, 146(1):145--55.

\bibitem[Kapulkin et~al., 2014]{KapuLenFanuVoev}
Kapulkin, C., Lumsdaine, P.~L., and Voevodsky, V. (2014).
\newblock The simplicial model of univalent foundations.
\newblock Preprint on the mathematics arXiv, 1211.2851v2.

\bibitem[Mathias, 2001]{MathThe}
Mathias, A. R.~D. (2001).
\newblock The strength of {M}ac {L}ane set theory.
\newblock {\em Annals of Pure and Applied Logic}, 110:107--234.
\newblock Archived at www.dpmms.cam.ac.uk/~ardm/maclane.pdf.

\bibitem[McLarty, 2014]{McLArXivLarge}
McLarty, C. (2014).
\newblock The large structures of {G}rothendieck founded on finite order
  arithmetic.
\newblock Preprint on the mathematics arXiv, 1102.1773v4.

\bibitem[Simpson, 1973]{SimpsonRevFelgner}
Simpson, S. (1973).
\newblock Review: {U}lrich {F}elgner, {C}omparison of the axioms of local and
  universal choice; {A}ndrzej {M}ostowski, models of second order arithmetic
  with definable {S}kolem functions.
\newblock {\em Journal of Symbolic Logic}, 38:652--53.

\bibitem[van~den Berg and Moerdijk, 2014]{vandenBergMoerdijk}
van~den Berg, B. and Moerdijk, I. (2014).
\newblock W-types in homotopy type theory.
\newblock Preprint on the mathematics arXiv:1307.2765v2.

\end{thebibliography}

\end{document}